\documentclass[review,12pt,3p]{elsarticle}
\usepackage{lineno,hyperref}
\modulolinenumbers[1]
\usepackage{graphicx}
\usepackage{amsmath,amsfonts}
\usepackage{amstext,amsthm,amssymb}
\usepackage{mathabx}
\usepackage{mathrsfs}
\usepackage{ragged2e}
\usepackage[mathscr]{euscript}
\usepackage{multirow}
\usepackage{subfigure}
\usepackage{enumerate}
\usepackage{mathrsfs}
\usepackage{color}
\usepackage{verbatim}
\usepackage{float}
\usepackage{framed}

\newtheorem{thm}{Theorem}[]
\newtheorem{lem}[thm]{Lemma}
\newtheorem{prop}[thm]{Prop.}

\newtheorem{rem}[thm]{Remark}

\graphicspath{{figures/}{figs/}}

\begin{document}
\begin{frontmatter}

\title{Continuous and discrete inf-sup conditions for surface incompressibility
of a deformable continuum}

\author[]{Gustavo C. Buscaglia}
\ead{gustavo.buscaglia@icmc.usp.br}
\address{Instituto de Ci\^encias Matem\'aticas e de Computa\c c\~ao, 
Universidade de S\~ao Paulo, Av. do Trabalhador s\~ao-carlense, 400,  
13560-970 S\~ao Carlos, SP, Brazil}


\begin{abstract}
Surface incompressibility, also called inextensibility, imposes a zero-surface-divergence constraint on the velocity of a closed deformable material surface. The well-posedness of the mechanical problem under such constraint depends on an inf-sup or stability condition for which an elementary proof is provided. The result is also shown to hold in combination with the additional constraint of preserving the enclosed volume, or isochoricity. These continuous results are then applied to prove a modified discrete inf-sup condition that is crucial for the convergence of stabilized finite element methods.
\end{abstract}

\end{frontmatter}

\linenumbers


\section{Introduction}

The rate of change of the volume $|\omega|$
of an infinitesimal piece of continuum $\omega$ located at point
${\bf x}$ at time $t$ is given by \cite{gurtin81}
\begin{equation}
\frac{d|\omega|}{dt}\,=\,|\omega|\,\nabla \cdot {\bf u} ({\bf x},t).
\end{equation}
There exist materials that preserve volume exactly, or within
experimental accuracy, and are modeled as {\em incompressible}.
%
The incompressibility constraint ($\nabla\cdot {\bf u}=0$) materializes in
the equations of motion as a reaction force, which is
the gradient of an unknown pressure field $p$. The
pressure field is a uniquely determined element of $Q=L^2(\Omega)$ because
the divergence operator $\nabla \cdot$ is {\em surjective} onto $Q$.
The surjectivity is equivalent to the inf-sup condition\cite{girault_raviart,brezzifortin}
\begin{equation}
\inf_{q\in Q}~\sup_{{\bf v}\in {\bf H}^1(\Omega)}~
\frac{\int_\Omega q\,\nabla\cdot {\bf v}}{\|{\bf v}\|_1\,\|q\|_0}
~>~0~.
\label{infsupvol}
\end{equation}
If the continuum is restricted at the boundary, more precisely,
if the velocity field normal to the boundary is constrained
all over $\partial\Omega$, then the mean value of the pressure
is undetermined. This reflects mathematically in that
to satisfy (\ref{infsupvol}) the pressure
space $Q$ must be chosen as a strict subspace of $L^2(\Omega)$,
for example $L^2_0(\Omega)$ which consists of functions with zero mean.

If we now consider a {\em smooth, closed, orientable surface} $\Gamma$ 
evolving in $\mathbb{R}^3$,
the analogous property to volume preservation is {\em area preservation},
also called {\em inextensibility} or {\em surface incompressibility},
which is indeed exhibited by many relevant materials. Our main interest
is in lipid membranes \cite{arroyo2,tb13_cmame,ramb15_jcp,bgn15_ta}, 
which are area-preserving surface fluids \cite{seifert97},
but the inextensibility constraint can also hold in other,
fluid or solid, material surfaces.

The rate of
change of the area $|\varpi|$ of an infinitesimal piece of surface $\varpi$ 
located at ${\bf x}$ at time $t$ and moving along a velocity field ${\bf u}$ 
is given by
\begin{equation}
\frac{d|\varpi|}{dt} = |\varpi|\,\nabla_\Gamma\cdot {\bf u} ({\bf x},t)
\end{equation}
where we have introduced the 
{\em surface divergence} operator $\nabla_\Gamma \cdot$. The velocity
field of an inextensible surface $\Gamma$ must thus satisfy
\begin{equation}
\nabla_\Gamma \cdot {\bf u} ({\bf x},t) = 0 \qquad 
\mbox{a.e. in}~\Gamma~.
\end{equation}

\begin{rem}
We adopt here the operators of tangential calculus
as presented by Delfour and Zol\'esio \cite{delfourzolesio11},
with the specific notation of Buscaglia and Ausas \cite{ba11_cmame}.
In particular, the surface gradient $\nabla_\Gamma f$ of a function
$f:\Gamma\to\mathbb{R}$ at the point ${\bf x}\,\in\,\Gamma\subset
\mathbb{R}^3$ is the three-dimensional vector
\begin{equation}
\nabla_\Gamma f\,({\bf x}) \doteq \nabla \widehat{f} ({\bf x})~,
\end{equation}
where 
\begin{equation}
\widehat{f}({\bf x})\doteq f(\Pi_\Gamma {\bf x})~,
\end{equation}
$\Pi_\Gamma {\bf x}$ being the normal projection of ${\bf x}\,\in\,\mathbb{R}^3$
onto $\Gamma$. Surface gradients of vector fields are computed
one Cartesian component at a time ($(\nabla_\Gamma {\bf u})_{ij}=(\nabla \widehat{u}_i)_j$).
\end{rem}


If the velocity of the surface ${\bf u}\,\in\,{\bf V}$ satisfies a variational
formulation
\begin{equation}
a({\bf u},{\bf v})=f({\bf v}) \qquad \forall\,{\bf v}\,\in\,{\bf V}~,
\end{equation}
for some continuous bilinear and linear forms $a(\cdot,\cdot)$ and
$f(\cdot)$, then the {\em mixed} formulation that enforces
the inextensibility constraint is: 

\begin{framed}
\noindent
{\bf Mixed inextensible formulation:} {\em Find $({\bf u},\sigma)\,
\in\,{\bf V}\times \Sigma$ such that
\begin{eqnarray}
a({\bf u},{\bf v}) + b({\bf v},\sigma) &=& f({\bf v})\qquad
\forall\, {\bf v}\,\in\,{\bf V} \label{eqmif1}\\
b({\bf u},\xi)&=&0\qquad\qquad\forall\,\xi\,\in\,\Sigma~, \label{eqmif2}
\end{eqnarray}
where
\begin{equation}
b({\bf v},\xi) \doteq \int_\Gamma \xi\,\nabla_\Gamma \cdot {\bf v}~.
\end{equation}
}
\end{framed}

The Lagrange multiplier $\sigma$ is the reaction to the inextensibility
constraint, a scalar field that is known as {\em surface tension}.
It is well known \cite{brezzifortin,ErnGuermond04} that a necessary condition for 
(\ref{eqmif1})-(\ref{eqmif2}) to be well posed, and in particular for
$\sigma$ to exist and be unique, is the inf-sup condition
\begin{equation}
\inf_{0\neq \xi\,\in\,\Sigma}~\sup_{0\neq {\bf v}\,\in\,{\bf V}}~
\frac{b({\bf v},\xi)}{\|{\bf v}\|_{\bf X} \|\xi\|_\Sigma} \doteq \alpha >0~.
\label{eqinfsup1}
\end{equation}
We assume that the velocity space ${\bf V}$ is a subspace of
\begin{equation}
{\bf X}=\{{\bf v}\,\in\,{\bf H}^1(\Gamma)~|~v_n\doteq {\bf v}\cdot 
\widecheck{\bf n} \,\in\,H^m(\Gamma) \}~,
\end{equation}
$\widecheck{\bf n}$ being the unit normal, with two possibilities:
\begin{itemize}
\item If $m=1$, we have ${\bf X}={\bf H}^1(\Gamma)$ and we set
\begin{equation}
\|{\bf v}\|_{\bf X} = \|{\bf v}\|_0 + \ell\,|{\bf v}|_1~.
\end{equation}
\item If $m>1$ the space ${\bf X}$ has extra regularity in the
normal component, and the norm is chosen as
\begin{equation}
\|{\bf v}\|_{\bf X} = \|{\bf v}\|_0 + \ell\,|{\bf v}|_1 + \ell^m\,
|v_n|_m~.
\end{equation}
\end{itemize}
Above, $\|\cdot\|_m$ (respectively, $|\cdot|_m$) denotes the usual
norm (respectively, seminorm) on $H^m(\Gamma)$. The same
notation is used, without risk of confusion, for the norm and
seminorm of ${\bf H}^m(\Gamma)$ (space of vector fields with components
in $H^m(\Gamma)$). The length scale $\ell$ is
a constant introduced to make the units consistent. The added regularity
in the normal component results, in the case of lipid membrane models, from
a curvature dependent energy \cite{canham,helfrich,deuling_helfrich,seifert97,sf14_jmb}. Our interest in this article lies in identifying 
appropriate combinations of spaces ${\bf V}\subset {\bf X}$ and $\Sigma$
so that (\ref{eqinfsup1}) is satisfied.

Notice that $b({\bf v},\sigma)$ in (\ref{eqmif1}) is the dynamical
action of $\sigma$. Using the integration by parts formula for closed
surfaces \cite{delfourzolesio11,ba11_cmame,de13_an}
\begin{equation}
\int_\Gamma \sigma\, \nabla \cdot {\bf v}
= \int_\Gamma {\bf v}\cdot  \left (-\nabla_\Gamma \sigma + H\,\sigma\,\widecheck{\bf n}
\right )\qquad\qquad\forall\,{\bf v}\,\in\,{\bf X},\qquad \forall\,\sigma\,
\in\,H^1(\Gamma)~,
\label{eqgauss}
\end{equation}
where $H=\nabla_\Gamma\cdot\widecheck{\bf n}$ is the mean curvature,
one recovers the classical expression of the surface tension force
\begin{equation}
{\bf f}_\sigma = \nabla_\Gamma \sigma - H\,\sigma\,\widecheck{\bf n}~.
\end{equation}

The inextensibility constraint is frequently imposed 
together with the constraint of {\em isochoricity},
that is, of {\em preserving the enclosed volume}.
Many physical situations admit such an idealization,
as for example the situation in which an impermeable material 
surface encloses an incompressible medium. In the case of
lipid membranes, isochoricity is a consequence of osmotic
equilibrium \cite{seifert97}.

The rate of change of the enclosed volume
$\mathcal{V}$ when $\Gamma$ moves along the velocity field ${\bf u}$ is
\begin{equation}
\frac{d\mathcal{V}}{dt}=\int_\Gamma {\bf u}\cdot\widecheck{\bf n}~.
\end{equation}
The Lagrange multiplier that enforces this constraint
turns out to be an {\em internal uniform
pressure} $p\,\in\,\mathbb{R}$
acting as a uniform normal
force $p\,\widecheck{\bf n}({\bf x})$ at each ${\bf x}\,\in\,\Gamma$. 

If just the isochoricity constraint is imposed,
the corresponding inf-sup condition is
\begin{equation}
\inf_{r\,\in\,\mathbb{R}}~\sup_{{\bf v}\in {\bf V}}~
\frac{r\,\int_\Gamma {\bf v}\cdot\widecheck{\bf n}}
{|r|\,\|{\bf v}\|_{\bf X}} ~>~0~.
\label{infsupp}
\end{equation}

If both the inextensibility and isochoricity constraints 
hold simultaneously, the mixed formulation becomes:

\begin{framed}
\noindent
{\bf Mixed inextensible-isochoric formulation:} 
{\em Find $({\bf u},\sigma,p)\,
\in\,{\bf V}\times \Sigma\times \mathbb{R}$ such that
\begin{eqnarray}
a({\bf u},{\bf v}) + c({\bf v},(\sigma,p)) &=& f({\bf v})\qquad
\forall\, {\bf v}\,\in\,{\bf V} \label{eqmiif1}\\
c({\bf u},(\xi,q))&=&0\qquad\qquad\forall\,(\xi,q)\,\in\,\Sigma\times\mathbb{R}~, 
\label{eqmiif2}
\end{eqnarray}
where
\begin{equation}
c({\bf v},(\xi,q)) \doteq \int_\Gamma \left ( 
\xi\,\nabla_\Gamma \cdot {\bf v} + q\,{\bf v}\cdot \widecheck{\bf n}
\right )
~.
\end{equation}
}
\end{framed}

In this case {\em both} its surface tension $\sigma$ 
(as a function of ${\bf x}\,\in\,\Gamma$)
and its internal pressure $p$ are uniquely defined in $\Sigma\times \mathbb{R}$,
under the following condition on ${\bf V}$-$\Sigma$,
\begin{equation}
\inf_{(\xi,q)\in \Sigma\times \mathbb{R}}~\sup_{{\bf v}\in {\bf V}}~
\frac{c({\bf v},(\xi,q))}{\|{\bf v}\|_{\bf X}\,\|(\xi,q)\|_{\Sigma\times \mathbb{R}}}
~=~\inf_{(\xi,q)\in \Sigma\times \mathbb{R}}~\sup_{{\bf v}\in {\bf V}}~
\frac{\int_\Gamma \left ( \xi\,\nabla_\Gamma \cdot {\bf v} +
\,q\,{\bf v}\cdot\widecheck{\bf n} \right )
}{\|{\bf v}\|_{\bf X}\,\|(\xi,q)\|_{\Sigma\times \mathbb{R}}}
~\doteq~\beta~>~0~.
\label{infsupqp}
\end{equation}

In what follows, we will select appropriate spaces $\Sigma$ and
prove the stability inequalities
(\ref{eqinfsup1}) and (\ref{infsupqp}) considering
two possibilities for ${\bf V}$: (a) the shape $\Gamma$ is fixed, implying
that ${\bf V}\subsetneq {\bf X}$ consists solely of tangential fields; and (b) the space ${\bf V}$ is unconstrained, i.e., ${\bf V}={\bf X}$.

\section{Inf-sup conditions for purely tangential motions}

If a vector field ${\bf v}\,\in\,{\bf X}$ is decomposed into its 
tangential component ${\bf v}_\tau$ and its normal part $v_n\,\widecheck{\bf n}$,
i.e.,
\begin{equation}
{\bf v}={\bf v}_\tau + v_n\,\widecheck{\bf n}~,
\label{eqdecomposition}
\end{equation}
its gradient takes the form
\begin{equation}
\nabla_\Gamma {\bf v} = \nabla_\Gamma {\bf v}_\tau + 
\widecheck{\bf n}\otimes \nabla_\Gamma v_n + v_n \nabla_\Gamma \widecheck{\bf n}
\label{eqgradient}
\end{equation}
and its surface divergence is given by
\begin{equation}
\nabla_\Gamma\cdot {\bf v} = \nabla_\Gamma\cdot {\bf v}_\tau
+ v_n\,\nabla_\Gamma \cdot \widecheck{\bf n}~.
\end{equation}
The last term in (\ref{eqgradient}) 
contains the {\em curvature tensor} 
\begin{equation}
{\bf H}=\nabla_\Gamma\widecheck{\bf n}~.
\end{equation}

If ${\bf V}$ only consists of tangential motions, then $v_n=0$
for all ${\bf v}\,\in\,{\bf V}$, the numerator inside the inf-sup in (\ref{infsupp})
vanishes identically and thus the condition is not satisfied. The internal
pressure is not uniquely defined, and modifying it has no effect on the
motion of the surface continuum.

The surface tension $\sigma$, on the other hand, is well defined in
$L^2(\Gamma)$ {\em up to an arbitrary additive constant}. It is uniquely
defined, for example, in
\begin{equation}
L^2_0(\Gamma)=\{q\,\in\,L^2(\Gamma)~|~\int_\Gamma q\,=\,0\}~.
\end{equation}

\begin{prop} If ${\bf V}$ is the closed subspace of ${\bf X}$
consisting of purely tangential motions, i.e.,
\begin{equation}
{\bf V} = \{{\bf v}\,\in\,{\bf X}~|~{\bf v}\cdot {\bf n}=0~~\mbox{a.e. in}~\Gamma
\}
\end{equation}
and $\Sigma=L^2_0(\Gamma)$, then the inf-sup condition (\ref{eqinfsup1}) holds
with
\begin{equation}
\alpha = \frac{1}{c_r\,\ell}
\end{equation}
where $c_r$ is the elliptic regularity constant (see (\ref{eqregularity}) below). 
\end{prop}

\begin{proof} This is proved in the same way as (\ref{infsupvol}) is proved.
Given $\xi$ arbitrary in $L^2_0(\Gamma)$, let 
$\varphi\,\in\,H^1(\Gamma)\cap L^2_0(\Gamma)$ be the solution of
\begin{equation}
\Delta_\Gamma \varphi = \xi~.
\end{equation}
Taking ${\bf v}=\nabla_\Gamma \varphi$, one has 
\begin{equation}
b({\bf v},\xi)=
\int_\Gamma \xi\,\nabla_\Gamma\cdot {\bf v} = \|\xi\|_0^2~.
\end{equation}
On the other hand, from the regularity estimate \cite{de13_an}
\begin{equation}
\|\varphi\|_0+\ell\,|\varphi|_1+\ell^2\,|\varphi|_2 \leq c_r\,\ell^2\,\|\xi\|_0~
\label{eqregularity}
\end{equation}
(where again $\ell$ is introduced to render the units consistent)
one has that 
$\|{\bf v}\|_{\bf X}=\|{\bf v}\|_0+\ell\,|{\bf v}|_1
=|\varphi|_1 + \ell\,|\varphi|_2 \leq c_r\,\ell\,\|\xi\|_0$, 
and thus the claim (\ref{eqinfsup1}) is proved.
\end{proof}

\section{The inf-sup condition for arbitrary motions}

If a surface can move along its normal direction
 then condition (\ref{infsupp}) is seen to hold simply taking 
${\bf v}=\widecheck{\bf n}$. The preservation of the enclosed volume uniquely
defines an internal pressure. 

Turning to the inextensibility condition, one has the following:

\begin{prop}
If ${\bf V}={\bf X}$ and $\Sigma=L^2(\Gamma)$, then (\ref{eqinfsup1})
holds.
\label{prop2}
\end{prop}

This means that the surface tension is {\em completely} defined on an
inextensible surface that can move freely in space. We will prove
this proposition after proving Prop. \ref{prop3}, since by then
it will become straightforward.

Consider now the case of a surface that {\em both} is inextensible
{\em and} preserves the enclosed volume. Rewriting the corresponding
inf-sup condition (\ref{infsupqp}) as
\begin{equation}
\inf_{(\xi,q)\in \Sigma\times\mathbb{R}}~\sup_{{\bf v}\in {\bf V}}~
\frac{\int_\Gamma \left [\xi\,\nabla_\Gamma \cdot {\bf v}_\tau + (H\xi+q)\,v_n\right ]}
{\|{\bf v}\|_{\bf X}\,\|(\xi,q)\|_{Q\times \mathbb{R}}}
~=~\beta~>~0~
\label{infsupq2}
\end{equation}
one immediately sees that these two conditions {\em are linearly dependent
if the surface has constant mean curvature} (a sphere). Taking $q=H$ (constant)
and $\xi=-1$ (constant), and remembering that {\em the integral of the surface
divergence of a purely tangential field over a closed surface vanishes}
(a consequence of (\ref{eqgauss}) taking $\sigma=1$), 
one has the numerator inside the inf-sup equal
to zero for all ${\bf v}$. For a sphere, thus, inextensibility {\em implies}
the preservation of the enclosed volume. The internal pressure and the
mean surface tension are not uniquely defined. The gradient of the surface
tension, on the other hand, is well determined. 

For all surfaces other than spheres, however, one has:

\begin{prop}
If the surface $\Gamma$ is not a sphere, then (\ref{infsupqp}) holds
with ${\bf V}={\bf X}$ and $\Sigma=L^2(\Gamma)$.
\label{prop3}
\end{prop}

The proof is quite straightforward, but let us for clarity first
state a couple of preliminary facts. Let $\overline{g}$ denote
the mean value of $g\,\in\,L^1(\Gamma)$.

\begin{lem}
The expression
\begin{equation}
\|(\xi,q)\|_{L^2(\Gamma)\times \mathbb{R}} =
\left [
\|\xi-\overline{\xi}\|_0^2 
+ \ell^4\,(\widebar{H}\, \overline{\xi}+q)^2
+ \ell^2\,\overline{{\xi}}^2
\right ]^{\frac12}
\end{equation}
is indeed a norm on $L^2(\Gamma)\times \mathbb{R}$. $\Box$
\end{lem}

As a consequence of (\ref{eqgradient}) one also has:

\begin{lem} 
\begin{equation}
\|{\bf v}\|_{0} + \ell\,|{\bf v}|_1 \leq
 \|{\bf v}_\tau\|_0 + \ell\,|{\bf v}_\tau|_1 + 
(1+\ell \|{\bf H}\|_{L^\infty(\Gamma)}) \|v_n\|_0 + \ell |v_n|_1
\label{eqlemma2}
\end{equation}
for all ${\bf v}\,\in\,{\bf H}^1(\Gamma)$.
\end{lem}

Now we proceed to prove Prop. \ref{prop3}.

\begin{proof}
Let $(\xi,q)$ be an arbitrary element of $\Sigma\times \mathbb{R}$. Taking
$\varphi\,\in\,H^1(\Gamma)\cap L^2_0(\Gamma)$ as the unique solution of $\Delta_\Gamma \varphi = \xi-\overline{\xi}$,
we select ${\bf v}\,\in\,{\bf V}$ as
\begin{equation}
{\bf v} = \nabla_\Gamma \varphi + v_n\,\widecheck{\bf n}~,
\qquad\mbox{with}~~
v_n = k_1 (\widebar{H}\,\overline{\xi}+q) 
+ k_2 (H-\widebar{H})\, \overline{\xi}~.
\label{eq22}
\end{equation}
The positive constants $k_1$ and $k_2$ are left unspecified for now.
Rewriting
\begin{equation}
c({\bf v},(\xi,q))=\int_\Gamma \left [
\xi\,\nabla_\Gamma \cdot {\bf v}_\tau + (H\xi+q)\,v_n \right ]
\end{equation}
and using (\ref{eq22}) one has
\begin{eqnarray}
c({\bf v},(\xi,q))& = &\|\xi-\overline{\xi}\|_0^2 
+ k_1\,\left [ |\Gamma|\,(\widebar{H}\,\overline{\xi}+q)^2
+ (\widebar{H}\,\overline{\xi}+q)\,\int_\Gamma (H,\xi-\widebar{H}\,\overline{\xi}) 
\right ] +\nonumber \\
& & + k_2 \int_\Gamma \left [
(H-\widebar{H})^2{\overline{\xi}}^2
+(H-\widebar{H})\,H\,\overline{\xi} (\xi-\overline{\xi}) 
+ (H-\widebar{H})\,(\widebar{H}\,\overline{\xi}+q)\,
\overline{\xi} \right ]~.
\end{eqnarray}
Noticing that the last term in the second integral cancels out and
using that $\int_\Gamma (H\xi-\widebar{H}\,\overline{\xi}) \leq \|H-\widebar{H}\|_0\,\|\xi-\overline{\xi}\|_0$,
one arrives at
\begin{eqnarray}
c({\bf v},(\xi,q)) & \geq & \|\xi-\overline{\xi}\|_0^2 + k_1\,|\Gamma|\,
(\widebar{H}\,\overline{\xi}+q)^2
+ k_2 \|H-\widebar{H}\|_0^2 \,{\overline{\xi}}^2 - \nonumber \\
& & - k_1 \|H-\widebar{H}\|_0\,|\widebar{H}\,\overline{\xi}+q|\,\|\xi-\overline{\xi}\|_0
- k_2 \|(H-\widebar{H})^2\|_0 \,|\overline{\xi}|\,\|\xi-\overline{\xi}\|_0
\end{eqnarray}
which using Young's inequality twice yields
\begin{eqnarray}
c({\bf v},(\xi,q)) & \geq & 
\left ( 1 - k_1 \frac{\|H-\widebar{H}\|_0^2}{2\,|\Gamma|}
- k_2 \frac{\|(H-\widebar{H})^2\|_0^2}{2\,\|H-\widebar{H}\|_0^2} \right ) \,
\|\xi-\overline{\xi}\|_0^2 +\nonumber \\
& & + k_1\,\frac{|\Gamma|}{2}\,(\widebar{H}\,\overline{\xi}+q)^2
+ k_2\,\frac{\|H-\widebar{H}\|_0^2}{2}\,{\overline{\xi}}^2~.
\label{eq38}
\end{eqnarray}

Let us now choose
\begin{eqnarray}
k_1 &=& \min \left \{ \frac{|\Gamma|}{2\|H-\widebar{H}\|_0^2}\,,\,\ell^2\right \}\\
k_2 &=& \min \left \{ \frac{\|(H-\widebar{H})\|_0^2}
{2\,\|(H-\widebar{H})^2\|_0^2}\,,\,\ell^2\right \}
\end{eqnarray}
 gives
\begin{eqnarray}
c({\bf v},(\xi,q)) & \geq & \frac12
\|\xi-\overline{\xi}\|_0^2 + k_1\,\frac{|\Gamma|}{2}\,(\widebar{H}\,\overline{\xi}+q)^2
+ k_2\,\frac{\|H-\widebar{H}\|_0^2}{2}\,{\overline{\xi}}^2 \label{eq41}\\
&\geq & A~~\|(\xi,q)\|_{L^2(\Gamma)\times \mathbb{R}}^2~,
\end{eqnarray}
with
\begin{eqnarray}
A&\doteq &
~\min~\left \{
\frac12\,,\,\frac{|\Gamma|^2}{4\,\ell^4\,\|H-\widebar{H}\|_0^2}\,,\,
\frac{|\Gamma|}{2\,\ell^2}\,,\,
\frac{\|H-\widebar{H}\|_0^4}{4\,\ell^2\,\|(H-\widebar{H})^2\|_0^2}\,,\,
\frac{\|H-\widebar{H}\|_0^2}{2}\right \}
~.
\end{eqnarray}
At the same time, from  (\ref{eqlemma2}) and the estimates
\begin{eqnarray}
\|{\bf v}_\tau\|_0+\ell\,|{\bf v}_\tau|_1 & \leq & c_r\,\ell\,\|\xi-\overline{\xi}\|_0 \\ & & \nonumber \\
\|v_n\|_0 &\leq &\ell^2\,|\Gamma|^{\frac12}\,|\widebar{H}\,\overline{\xi}+q|~+~
\ell^2\,\|H-\widebar{H}\|_0\,|\overline{\xi}|\\
& & \nonumber\\
|v_n|_k &\leq &\ell^2\,|H|_k\,|\overline{\xi}| \qquad \forall\,k\geq 1
\end{eqnarray}
it follows that 
\begin{eqnarray}
\|{\bf v}\|_{\bf X} &\leq &c_r\,\ell\,\|\xi-\bar{\xi}\|_0
+ (1+\ell\,\|{\bf H}\|_{L^\infty(\Gamma)})\,\ell^2\,|\Gamma|^{\frac12}
~\left | \widebar{H}\,\overline{\xi}+q\right |+
\nonumber\\
& &+ \left [
(1+\ell\,\|{\bf H}\|_{L^\infty(\Gamma)})\,\ell^2\,\|H-\widebar{H}\|_0+\ell^3\,|H|_1+\ell^{m+2}
\,|H|_m\right ]~|\overline{\xi}|\nonumber \\ 
& \leq & B~~\|(\xi,q)\|_{L^2(\Gamma)\times \mathbb{R}}
\end{eqnarray}
with
\begin{equation}
B^2\doteq
c_r^2\,\ell^2+
(1+\ell\,\|{\bf H}\|_{L^\infty(\Gamma)})^2\,|\Gamma|
+ \left [
(1+\ell\,\|{\bf H}\|_{L^\infty(\Gamma)})\,\ell\,\|H-\widebar{H}\|_0+\ell^2\,|H|_1+\ell^{m+1}
\,|H|_m\right ]^2
\end{equation}
The claim is proved with $\beta=A/B$.
\end{proof}

Notice that $A$ is equal to zero for a sphere, and thus $\beta=0$.
However, Prop. \ref{prop2} is true irrespective of $\Gamma$ being a
sphere or not. Let us modify the previous proof to prove it.

\begin{proof} (of Prop. \ref{prop2}) Taking the same ${\bf v}$ as
before, and particularizing (\ref{eq41})
to $q=0$ one gets
\begin{equation}
b({\bf v},\xi)\geq
\frac12 \|\xi-\overline{\xi}\|_0^2 
+ \left (
k_1\,\frac{|\Gamma|\,\widebar{H}^2}{2}+ k_2\,\frac{\|H-\widebar{H}\|_0^2}{2}
\right )\,{\overline{\xi}}^2~\geq~C\,\|\xi\|_0^2~.
\end{equation}
with
\begin{equation}
C = \min \left \{
\frac12\,,\,
\frac{k_1\,\widebar{H}^2}{2}+\frac{k_2\,\|H-\widebar{H}\|_0^2}{2\,|\Gamma|}
\right \}~.
\end{equation}
and now $C>0$ even if $H=\widebar{H}$. Proposition \ref{prop2} is thus
proved with $\Sigma=L^2(\Gamma)$ and $\alpha=C/B$.

\end{proof}

\section{Discrete inf-sup condition}

The continuous inf-sup conditions proved above also allow for
the extension to surface finite elements of the discrete
counterpart known as Verf\"urth's lemma \cite{fhs93}. It is
central in the numerical analysis of stabilized finite element
methods such as the Galerkin-Least-Squares method \cite{fh87,ff92}.
A variant of one such method, the pressure gradient projection
method \cite{cb97,bbc00,cbbh01}, has recently been successfully
implemented for lipid membrane models \cite{tb13_cmame,ramb15_jcp}.

Let ${\bf V}_h\subset {\bf V}$ and $\Sigma_h\subset \Sigma$ be
surface finite element spaces, as defined in Dziuk and Elliott
\cite{de13_an}. Notice that these are {\em lifted} spaces, 
which are defined with
the aid of a faceted surface but consist of 
scalar functions (or vector fields) 
defined on the ``exact'' surface $\Gamma$. 

\begin{prop} \label{prop7}If the space $\Sigma_h$ consists of continuous
functions, then there exist $\gamma>0$ and $\delta>0$, independent of the
mesh size $h\doteq \max_K h_K$, such that
\begin{equation}
\sup_{{\bf v}_h\in {\bf V}_h}~
\frac{\int_\Gamma (\xi_h\nabla_\Gamma\cdot{\bf v}_h
+ q\,{\bf v}_h\cdot \widecheck{\bf n} )}
{\|{\bf v}_h\|_{\bf X}}
\geq \gamma\,\|(\xi_h,q)\|_{\Sigma\times \mathbb{R}}
- \delta \left (
\sum_{K\,\in\,\mathcal{T}_h}\, h_K^2\,\|\nabla_\Gamma \xi_h\|_{{\bf L}^2(K)}^2
\right )^{\frac12}
\label{eqprop7}
\end{equation}
for all $(\xi_h,q)\,\in\,\Sigma_h\times \mathbb{R}$.
\end{prop}

The proof assumes the existence of a Cl\'ement-type quasi-interpolatory
operator $\mathcal{I}_h:{\bf V}\to {\bf V}_h$ satisfying 
\begin{eqnarray}
\|\mathcal{I}_h{\bf v}\|_{\bf X}&\leq& c_1\,\|{\bf v}\|_{\bf X} \\
\|{\bf v}-\mathcal{I}_h{\bf v}\|_{{\bf L}^2(K)}&\leq
& c_2\,h_K\,\|\nabla_\Gamma {\bf v}\|_{{L}^2(\omega_K)}
\end{eqnarray}
where $\omega_K$ is the union of all elements that share at least one
node with element $K$ \cite{clement75,ErnGuermond04}.

\begin{proof}
For arbitrary $\xi_h$ and $q$, from Prop. \ref{prop3} we have
\begin{eqnarray}
\beta\,\|(\xi_h,q)\|_{\Sigma\times \mathbb{R}}
\leq
\sup_{{\bf v}\in {\bf V}} 
\frac{c(\mathcal{I}_h{\bf v},(\xi_h,q))}{\|{\bf v}\|_{\bf X}}
~+~\sup_{{\bf v}\in {\bf V}} 
\frac{c({\bf v}-\mathcal{I}_h{\bf v},(\xi_h,q))}{\|{\bf v}\|_{\bf X}}~.
\label{eq57}
\end{eqnarray}
It is clear that 
\begin{equation}
\sup_{{\bf v}\in {\bf V}} 
\frac{c(\mathcal{I}_h{\bf v},(\xi_h,q))}{\|{\bf v}\|_{\bf X}}
\leq~
c_1\,\sup_{{\bf v}\in {\bf V}} 
\frac{c(\mathcal{I}_h{\bf v},(\xi_h,q))}{\|\mathcal{I}_h{\bf v}\|_{\bf X}}
\leq~
c_1\,\sup_{{\bf v}_h\in {\bf V}_h} 
\frac{c({\bf v}_h,(\xi_h,q))}{\|{\bf v}_h\|_{\bf X}}~.
\end{equation}
Concerning the second term in the right-hand side of (\ref{eq57}),
and denoting ${\bf w}={\bf v}-\mathcal{I}_h{\bf v}$, we have
\begin{eqnarray}
c({\bf v}-\mathcal{I}_h{\bf v},(\xi_h,q))
& = & \sum_{K\,\in\,\mathcal{T}_h}
\int_K \left [
\xi_h\nabla_\Gamma\cdot {\bf w}+qH\,{\bf w}\cdot \widecheck{\bf n}
\right ] \nonumber \\
& = & \sum_{K\,\in\,\mathcal{T}_h} \int_{\partial K} 
\xi_h\,{\bf w}\cdot \widecheck{\boldsymbol{\nu}}
~+~
\sum_{K\,\in\,\mathcal{T}_h} \int_K
\left ( -\nabla_\Gamma\xi_h\cdot {\bf w}+
\xi_hH {\bf w}\cdot \widecheck{\bf n} \right )~.
\end{eqnarray}
The first sum vanishes if $\xi_h$ is continuous across
element boundaries and $\Gamma$ is $\mathscr{C}^1$.
As a consequence,
since $\|{\bf w}\|_{{\bf L}^2(K)} \leq 
c_2\,h_K\,\|\nabla_\Gamma {\bf v}\|_{{L}^2(\omega_K)}$,
\begin{eqnarray}
c({\bf v}-\mathcal{I}_h{\bf v},(\xi_h,q)) & \leq &
\frac{c_3}{\ell}\,
\left [
\left (
\sum_{K\,\in\,\mathcal{T}_h}
h_K^2\,
\|\nabla_\Gamma \xi_h\|_{{\bf L}^2(K)}^2 \right )^{\frac12}
+ h\,\|H\|_{L^\infty(\Gamma)}\,
\|\xi_h\|_0\right ]~~\|{\bf v}\|_{\bf X}
\end{eqnarray}
where
\begin{eqnarray}
c_3 & = & c_2\,\ell\,\sup_{{\bf v}\in {\bf V}} 
\frac{
\left (
\sum_{K\,\in\,\mathcal{T}_h} \|\nabla_\Gamma {\bf v}\|_{{L}^2(\omega_K)}^2
\right )^{\frac12}}{\|{\bf v}\|_{\bf X}}~.
\end{eqnarray}
Replacing into (\ref{eq57}) one gets
\begin{eqnarray}
\beta\|(\xi_h,q)\|_{\Sigma\times \mathbb{R}}
&\leq& 
c_1~\sup_{{\bf v}_h\in {\bf V}_h} 
\frac{c({\bf v}_h,(\xi_h,q))}{\|{\bf v}_h\|_{\bf X}}~
+ \frac{c_3}{\ell}\,\left (
\sum_{K\,\in\,\mathcal{T}_h}
h_K^2\,
\|\nabla_\Gamma \xi_h\|_{{\bf L}^2(K)}^2 \right )^{\frac12}
+ \frac{c_3\,h\,\|H\|_{L^\infty(\Gamma)}}{\ell}\,
\|\xi_h\|_0 \nonumber \\
& &
\end{eqnarray}
which proves (\ref{eqprop7}) taking $\gamma=\beta/(2\,c_1)$,
$\delta=c_3/(\ell c_1)$ and $h$ small enough.
\end{proof}

\section{Concluding remarks}

It has been shown that the inextensibility of a surface continuum, 
analogous to the incompressibility of a volumetric medium,
is a well-posed constraint for closed surfaces evolving 
in $\mathbb{R}^3$. It gives rise to a surface tension field $\sigma$
that is uniquely defined in $L^2(\Gamma)$.

The simultaneous imposition of both the inextensibility constraint and
the isochoricity constraint (preservation of the enclosed
volume) is also well posed, the only exception being that of
a spherical configuration of the surface continuum. In all other
cases, both the surface tension field $\sigma\,\in\, L^2(\Gamma)$
and the internal pressure $p\,\in\,\mathbb{R}$ are uniquely
determined by the constrained problem.

The estimates in the proofs require the surface $\Gamma$ to
be of class $\mathscr{C}^2$ and have its mean curvature $H$
in $H^m(\Gamma)$, $m\geq 1$.

On the basis of the exact well-posedness, a discrete
stability result was established for discretizations of the
surface tension field that consist of continuous interpolants.
It consists of a modified inf-sup condition (sometimes called
Verf\"urth's lemma) which plays a central role in the justification
of stabilized methods for incompressible flow. The extension 
of this inf-sup condition to deformable surfaces thus justifies
the stabilized treatment of the surface tension proposed in
recent work on lipid membranes \cite{tb13_cmame,ramb15_jcp}.

\section*{ACKNOWLEDGMENTS}

The authors gratefully acknowledge the financial support received 
from S\~ao Paulo Research Foundation (FAPESP) 
and from the Brazilian National Research and Technology Council (CNPq).
Thanks are also due to D. Rodrigues for a careful revision of the manuscript.

\end{document}